\documentclass{compositio}

\input xy
\xyoption{all}
\usepackage{epsfig}
\usepackage{color}
\usepackage{amsthm}
\usepackage{amssymb}
\usepackage{amsmath}
\usepackage{amscd}
\usepackage{amsopn}
\usepackage{enumerate}
\usepackage{url}

\usepackage{hyperref}\hypersetup{colorlinks}

\usepackage{color} 

\definecolor{darkred}{rgb}{1,0,0} 
\definecolor{darkgreen}{rgb}{0,0.8,0}
\definecolor{darkblue}{rgb}{0,0,1}

\hypersetup{colorlinks,
linkcolor=darkblue,
filecolor=darkgreen,
urlcolor=darkred,
citecolor=darkgreen}

\newcommand{\labell}[1] {\label{#1}}
\newtheorem{thm}{Theorem}[section] 
\newtheorem{lem}{Lemma}[section] 
\newtheorem{prop}{Proposition}[section]
\newtheorem{cor}{Corollary}[section]
\theoremstyle{definition} 
\theoremstyle{remark} 
\newtheorem{rem}{Remark}

\expandafter\chardef\csname pre amssym.def at\endcsname=\the\catcode`\@
\catcode`\@=11
\def\undefine#1{\let#1\undefined}
\def\newsymbol#1#2#3#4#5{\let\next@\relax
 \ifnum#2=\@ne\let\next@\msafam@\else
 \ifnum#2=\tw@\let\next@\msbfam@\fi\fi
 \mathchardef#1="#3\next@#4#5}
\def\mathhexbox@#1#2#3{\relax
 \ifmmode\mathpalette{}{\m@th\mathchar"#1#2#3}%
 \else\leavevmode\hbox{$\m@th\mathchar"#1#2#3$}\fi}
\def\hexnumber@#1{\ifcase#1 0\or 1\or 2\or 3\or 4\or 5\or 6\or 7\or 8\or
 9\or A\or B\or C\or D\or E\or F\fi}

\font\teneufm=eufm10
\font\seveneufm=eufm7
\font\fiveeufm=eufm5
\newfam\eufmfam
\textfont\eufmfam=\teneufm
\scriptfont\eufmfam=\seveneufm
\scriptscriptfont\eufmfam=\fiveeufm

\catcode`\@=\csname pre amssym.def at\endcsname

\newcommand{\CR}{{\mathcal R}}

\newcommand{\Ff}{{\mathcal F}}

\def    \R      {{\mathbb R}}

\def    \Z      {{\mathbb Z}}
\def    \N      {{\mathbb N}}

\def    \T      {{\mathbb T}}
\def    \CP     {{\mathbb C}{\mathbb P}}

\def    \12    {{\frac{1}{2}}}

\def    \rk     {\operatorname{rk}}

\def    \HF     {\operatorname{HF}}

\def    \CF     {\operatorname{CF}}

\def    \Fix     {\operatorname{Fix}}
\def    \Per     {\operatorname{Per}}

\def    \va     {\vec{a}}

\begin{document}


\title[A symplectic proof of a theorem of Franks]{A symplectic proof of a theorem of Franks}

\author[B. Collier]{Brian Collier}
\email{collier3@illinois.edu}
\address{Department of Mathematics, University of Illinois at 
Urbana-Champaign, Urbana, IL 61801, USA}
\author[E. Kerman]{Ely Kerman}
\email{ekerman@illinois.edu}
\address{Department of Mathematics, University of Illinois at 
Urbana-Champaign, Urbana, IL 61801, USA}
\author[B. Reiniger ]{Benjamin M. Reiniger}
\email{reinige1@illinois.edu}
\address{Department of Mathematics, University of Illinois at 
Urbana-Champaign, Urbana, IL 61801, USA}
\author[B. Turmunkh]{Bolor Turmunkh}
\email{turmunk1@illinois.edu}
\address{Department of Mathematics, University of Illinois at 
Urbana-Champaign, Urbana, IL 61801, USA}
\author[A. Zimmer]{Andrew Zimmer}
\email{aazimmer@umich.edu}
\address{Department of Mathematics, University of Illinois at 
Urbana-Champaign, Urbana, IL 61801, USA}
\curraddr{Department of Mathematics, University of Michigan, Ann Arbor, MI 48109-1043, USA}

\shortauthors{B. Collier, E. Kerman, B.M. Reiniger, B. Turmunkh and A. Zimmer }

\classification{53D40, 37J45, 70H12}
\keywords{Periodic orbits, Hamiltonian flows, Floer homology}
\thanks{The authors acknowledge support from National Science Foundation grant DMS 08-38434 ÓEMSW21-MCTP: Research Experience for Graduate Students.}

\begin{abstract} 
A celebrated theorem in two-dimensional dynamics due to John Franks asserts that every area preserving homeomorphism of the sphere
has either two or infinitely many periodic points. In this work we reprove Franks'Ê theorem under the 
additional assumption that the map is smooth.  
Our proof uses only tools from symplectic topology and thus  differs significantly from previous proofs.  A crucial role is played by the results of Ginzburg and Kerman concerning resonance relations for Hamiltonian diffeomorphisms.

\end{abstract}

\maketitle

\section{Introduction}
\labell{sec:main-results}

Consider the unit sphere  $S^2 \subset \R^3$ equipped with the standard area form $\omega$ inherited from $\R^3$.
Let $\phi_{\alpha}$ be the rotation of the sphere by $2\pi \alpha$ radians  about the vertical axis.  Each $\phi_{\alpha}$ is an area preserving diffeomorphism
and there are two simple alternatives for the number of its periodic points: either $\alpha$ is irrational and  $\phi_{\alpha}$ has exactly two periodic points, the poles; or $\alpha$ is rational in which case some iterate of $\phi_{\alpha}$ is the identity and hence $\phi_{\alpha}$ has infinitely many periodic points. 
The following remarkable theorem due to  John Franks, \cite{f1, f2}, proves that these alternatives for the number of periodic points of area preserving maps of $S^2$ are universal.
\begin{thm}\label{franks}(\cite{f1,f2})
Every area preserving homeomorphism  of $S^2$  has either two or infinitely many periodic points.
\end{thm}

In the case of smooth maps this theorem was strengthened by Franks and Handel in \cite{fh} with the addition of new information on the growth rate of periodic points. The smoothness condition in \cite{fh} was then relaxed by Le Calvez in \cite{lc}.
As stated and proved, these results all belong to the world of  two-dimensional dynamical systems. In particular, all the previous proofs  known to the present authors utilize results such as Brouwer's translation theorem which capture  phenomena unique to dimension two. On the other hand, Franks' theorem (in the smooth category),  and the results in \cite{fh} can be recast as statements about Hamiltonian diffeomorphisms of $S^2$. From this perspective they can viewed as 
the two-dimensional models of a more general class of  results that are expected to hold for Hamiltonian diffeomorphisms of large families of  symplectic manifolds (see below).  

A first step in the process of absorbing Franks' theorem into symplectic topology is to reprove it using only the tools from this field.  This is the goal of the present paper. Other {\it symplectic} approaches to a similar set of results concerning area preserving disc maps  have been developed by  Bramham (see \cite{bh}) and by  Ghrist,  Van den Berg, Vandervorst and W\'ojcik in \cite{gvvw}.   

Here, we use some well-known symplectic tools, as well as the results on resonance relations for Hamiltonian diffeomorphisms from \cite{gk}, to prove the following.

\begin{thm} \label{main}
Every Hamiltonian diffeomorphism $\phi$ of $(S^2, \omega)$ has either two or  infinitely many periodic points. 
If $\phi$ has  exactly two periodic points, $P$ and $Q$, then both  are nondegenerate. In particular both are elliptic fixed points of $\phi$,  and their  mean indices, $\Delta(P)$, $\Delta(Q) \in \R / 4\Z$,  are irrational and satisfy $$\Delta(P)+ \Delta(Q) =0 \mod 4.$$  
\end{thm}

\noindent We defer a discussion of the mean index to Section \ref{meanindex}.

Theorem \ref{main} implies Theorem \ref{franks} in the 
smooth category. For the case of Hamiltonian diffeomorphisms with exactly two periodic points, the restrictions on these points included in the statement of Theorem \ref{main} are not new. As pointed out to us by Bramham, nondegeneracy follows from the results on area preserving homeomorphisms of annuli in \cite{f1,f2},  and the restrictions on the mean indices can be derived  from the Lefschetz fixed point theorem and the Poincar\'{e}-Birkhoff theorem. However, we establish these restrictions by other means which we hope will lead to analogous restrictions in some of the generalizations of Franks' theorem that are expected to hold in higher dimensional settings.  
Perhaps the best known (and most approachable) of these conjectured generalizations is the assertion that every Hamiltonian  diffeomorphism  of $\CP^n$ must have either $n+1$ or infinitely many periodic points  (see, for example, page 263 of \cite{hz}). 
Applications of the ideas developed here to such problems will be considered elsewhere. 

\begin{rem} Much is known about the set of Hamiltonian diffeomorphisms (in fact homeomorphisms) of $S^2$ with exactly two periodic points. In \cite{lc-icm} it is shown that,  up to conjugacy, every such Hamiltonian diffeomorphism (in fact homeomorphism) is the compactification of an irrational pseudo-rotation, an area preserving map of the closed annulus such that every positively recurrent point has the same irrational rotation number. On the other hand, such maps are known to exhibit a variety of different behaviors; from simple irrational rotations, to the smooth examples from \cite{ak, fh, fk}  which have only three ergodic invariant measures.
\end{rem}

\subsection{On the proof of Theorem \ref{main}}

The crucial  first step is to prove that if a Hamiltonian diffeomorphism $\phi$ of $S^2$ has finitely many periodic points, then at least two of these points, say $P$ and $Q$, must  have irrational mean indices.  
This is established as an essentially immediate implication of the theory of resonance relations for Hamiltonian diffeomorphisms developed in \cite{gk}. It is important to note that these results
from \cite{gk} are themselves implied by the ideas inherent in the recent proofs of the Conley Conjecture by Hingston in \cite{Hi} and Ginzburg in \cite{Gi:conley}, and the applications and refinements of these ideas 
developed by Ginzburg and G\"urel in \cite{gg1,gg2}. With $P$ and $Q$ in hand it is then easy to show that in order to prove Theorem \ref{main} it suffices to show that $\phi$ can not have another periodic point, say $R$, with an integer mean index.
Assuming the  existence of such an $R$, in two distinct cases, we blow up a suitable iteration of $\phi$ at two points,  and glue the resulting map to itself to obtain an area preserving diffeomorphism of the torus (following Arnold). Using index relations and the Floer theory of symplectic diffeomorphisms we then prove that the resulting maps can not exist.

\subsection{Acknowledgments} The second author is grateful to Viktor Ginzburg for 
many valuable comments and suggestions. He also wishes to thank Barney Bramham for
his generous and illuminating comments concerning an earlier version of this work.

\section{Background material, definitions and conventions}

\subsection{Symplectic isotopies and Hamiltonian diffeomorphisms} 
Let $(M, \omega)$  be a closed symplectic manifold of dimension $2n$ and minimal Chern number $N$. Our basic object of study will be a smooth isotopy $\psi_t$ 
of symplectic diffeomorphisms of $(M, \omega)$, where $t$ takes values in $[0,1]$ and $\psi_0$ is the identity map. 
In particular, we will be interested in the periodic points of $\psi_1$. Denoting the set of fixed points of $\psi_1$  by $\Fix(\psi_1)$, 
the set of periodic points of $\psi_1$ is defined as $$\Per(\psi_1) = \bigcup_{k \in \N} \Fix((\psi_1)^k).$$ 
The period of a point $X \in \Per(\psi_1)$ is defined to be the smallest positive integer $k$ for which $X \in \Fix((\psi_1)^k)$. 
We will also associate to each periodic point $X$ of $\psi_1$ with period $k$ the unique element of $\pi_1(M)$ or $\mathrm{H}_1(M;\Z)$ represented by the closed loop $t \mapsto (\psi_t)^k(X).$

To facilitate our study of periodic points, we will assume from now on that the time-dependent vector field $X_t$ generating our symplectic isotopy $\psi_t$ extends to a
smooth time-periodic vector field of period one. This imposes no new restrictions as any symplectic isotopy is homotopic, relative its endpoints, to one with this property. (In particular, 
$\psi_t$ is homotopic to $\psi_{\zeta(t)}$ where $\zeta \colon [0,1] \to [0,1]$ is smooth, nondecreasing, onto and constant near $0$ and $1$.) This assumption allows us to 
define  $\psi_t$ for all $t \in \R$ and to identify $\psi_{k}$ with $(\psi_1)^k$.

The subset of symplectic isotopies we are most interested in are those corresponding to Hamiltonian flows.
 A Hamiltonian on $(M, \omega)$ is a function $H\colon \R/\Z\times M\to \R$, or equivalently 
a smooth one-periodic family of functions $H_t(\cdot) = H(t, \cdot)$. Each Hamiltonian determines a one-periodic
 vector field $X_H$ on $M$ via the equation
$i_{X_H} \omega = -dH_t$. The time-$t$ flow of $X_H$, denoted by $\phi^t_H$, is defined for all $t \in \R$.
For $t \in [0,1]$, $\phi^t_H$ is a smooth isotopy of symplectic diffeomorphisms.
The set of Hamiltonian diffeomorphisms of $(M, \omega)$ consists of all the time one maps $\phi=\phi^1_H$ 
of Hamiltonian flows.

\subsection{The Conley-Zehnder and mean indices}\label{meanindex}

Let $A\colon [0,1] \to Sp(n)$ be a continuous path in the group $Sp(n)$ of $2n \times 2n$ symplectic matrices such that $A(0)$ is the identity matrix. 
One can associate to $A$ its Conley-Zehnder index $\mu(A) \in \Z$ 
as defined in \cite{cz3}, and its  mean index $\Delta(A) \in \R$ as defined in \cite{SZ}.  
As shown in \cite{SZ}, these indices satisfy the inequality
\begin{equation}
\label{close}
|\mu(A) - \Delta(A)| \leq n, 
\end{equation}
where the strict form of the inequality holds if $A(1)$ has at least one eigenvalue different from $1$.

Consider a smooth isotopy $\psi_t$ of symplectic diffeomorphisms as above. Let  $X$ be a fixed point of $\psi_1$ and let  $x \colon [0,1] \to M$
be the closed curve $\psi_t(X)$. Given a symplectic trivialization $\xi$ of $x^*TM$, the linearized flow of $\psi_t$ along $x(t)$ yields a smooth path $A_{\xi}\colon [0,1] \to Sp(n)$ 
starting at the  identity matrix. The quantities $\mu(A_{\xi})$ and $\Delta(A_{\xi})$ depend only on the homotopy class of the symplectic trivialization $\xi$. We denote this class by $[\xi]$
and define the Conley-Zehnder and mean index of $X$ with respect to this choice as
\begin{equation*}
\label{ }
\mu(X;\psi_t, [\xi]) = \mu(A_{\xi})
\end{equation*} 
 and 
\begin{equation*}
\label{ }
\Delta(X;\psi_t, [\xi]) = \Delta(A_{\xi}).\footnote{The symplectic isotopy is included in this notation because we will need to consider fixed points shared by different symplectic diffeomorphisms.} 
\end{equation*}
In this context, inequality \eqref{close} becomes
\begin{equation}
\label{closer}
|\mu(X;\psi_t, [\xi]) - \Delta(X;\psi_t, [\xi])| \leq n,
\end{equation}
where the strict form of the inequality holds if the linearization of $\psi_1$ at $X$, $D(\psi_1)_X$, has at least one eigenvalue different from $1$.

\subsubsection{Iteration formula}\label{iteration}

Each fixed point $X$ of $\psi_1$ is also a fixed point of the $k$-th iteration $(\psi_1)^k =\psi_k$. 
As shown in \cite{SZ}, the mean index grows linearly under iteration, i.e., 
\begin{equation}
\label{iter}
\Delta(X; \psi_{tk}, [\xi^k]) = k \Delta(X;\psi_t, [\xi]),
\end{equation}
where $\xi^k$ is the trivialization of $TM$ along $\psi_{tk}(X)$ induced by $\xi$.

\subsubsection{Continuity}\label{continuity}
We now recall a  continuity property of the mean index  established in \cite{SZ}. To do so we first note that if two fixed points $X$ and $X'$, of possibly different maps $\psi_1$ and $\psi'_1$,  represent the same homotopy class $c \in \pi_1(M)$ then we can specify a unique class of symplectic trivializations along both their trajectories by choosing a homotopy class of symplectic trivializations of $z^*(TM)$ where $z \colon S^1 \to M$ is any smooth representative of $c$ (in \cite{BU} such a choice is referred to as a $c$-structure.) In particular,  a choice of $[\xi]$ for $X$ determines a unique class of symplectic trivializations for $X'$ which we still denote by $[\xi]$. When we compare indices of fixed points  in the same homology class we will always assume that the classes of trivializations being used  are coupled in this manner.

Now let $\widetilde{\psi}_t$ be a symplectic isotopy $C^1$-close $\psi_t$. 
Under this perturbation, each fixed point $X$ of $\psi_1$ splits into a collection of fixed points of $\widetilde{\psi}_1$ which are close to $X$ (and hence in the same homotopy class as $X$). 
If $\widetilde{X}$ is one of these  fixed points of $\widetilde{\psi}_1$ then $$|\Delta(X; \psi_t, [\xi]) - \Delta(\widetilde{X}; \widetilde{\psi}_t, [\xi])|$$ is small.

\subsubsection{A useful fact in dimension two}

The following result is a simple consequence of the definition of the indices. It can be derived, for example,  from Theorem 7 in Chapter 8 of \cite{lo}. 

\begin{lem}\label{determine}
 Let $(M,\omega)$ be a  two-dimensional symplectic manifold  and suppose that  $\psi_t$ is an isotopy of symplectic diffeomorphisms of $(M,\omega)$ starting at the identity. If $X$ is a 
fixed point of $\psi_1$ and $\Delta(X;\psi_t, [\xi])$ is not an integer, then $\mu(X;\psi_t, [\xi])$ is the odd integer closest to $\Delta(X;\psi_t, [\xi])$. 
\end{lem}

\subsubsection{Indices of contractible fixed points modulo $2N$.} \label{indexmod}
When $X$ is a contractible fixed point of $\psi_t$, that is $x(t)=\psi_t(X)$ is contractible,
it is often useful to restrict attention to  trivializations of $x^*TM$ determined by a choice of smooth  spanning disc $u \colon \mathbb{D}^2 \to M$
with $u(e^{2\pi i t}) = x(t)$. For such choices of trivializations the corresponding  indices are well-defined
modulo twice the minimal Chern number, $2N$. In fact, the corresponding elements of $\R/2N\Z$ depend only on the time one map $\psi_1$ and hence will be denoted by $\mu(X)$ and $\Delta(X)$. 
The quantities $\Delta(P)$ and $\Delta(Q)$ appearing in the statement of Theorem \ref{main}
are meant to be understood in this way. 

\subsection{Floer homology for symplectic diffeomorphisms of the torus}\label{floerhomology}
Finally, we recall  the properties of the Floer homology of symplectic diffeomorphisms  required
for the proof of Theorem \ref{main}. We will only need to consider  the special case when $(M,\omega)$ is a two-dimensional symplectic torus $(\T^2, \Omega)$ and the symplectic diffeomorphism is isotopic to the identity.
Consider then  a smooth isotopy  $\psi_t$ of symplectic diffeomorphisms of $(\T^2, \Omega)$ starting at the identity such that the fixed points of $\psi_1$ are all nondegenerate. 
The Floer homology of $\psi_1$, $\HF(\psi_1)$, is then well-defined and has the properties described below. The reader is referred  to  \cite{ds, se1, se2, se3, va} for more details on the general construction  of this Floer homology,  and to \cite{c1,c2}  for more thorough reviews of the Floer theory of symplectic diffeomorphisms of surfaces.

\medskip

\noindent{\bf Invariance under Hamiltonian isotopy:} If $\phi$ is a Hamiltonian diffeomorphism of  $(\T^2, \Omega)$
  then \begin{equation*}
\label{ }
\HF(\psi_1 \phi)= \HF(\psi_1).
\end{equation*}
  
  \medskip

\noindent{\bf Splitting:} The Floer homology $\HF(\psi_1)$ admits a decomposition of the form 
\begin{equation*}
\label{ }
 \HF(\psi_1) = \bigoplus_{c \in \mathrm{H}_1(\T^2;\Z)} \HF(\psi_1; c).
\end{equation*}
Here, each summand $\HF(\psi_1; c)$ is the homology of a chain complex
$(\CF(\psi_1; c), \partial_J)$ where the chain group $\CF(\psi_1; c)$ is a torsion-free module over a suitable Novikov ring, and the rank of this module is the number of  fixed points  of $\psi_1$ which represent the class $c$.  The group $\HF(\psi_1; 0)$ coincides with the Floer-Novikov Homology constructed by L{\^e} and Ono in \cite{va}, (\cite{se3}).  Moreover, if $\psi_t$ is a Hamiltonian isotopy then $\HF(\psi_1) = \HF(\psi_1; 0)$ and is equal to the usual Hamiltonian Floer homology of $(\T^2, \Omega)$, \cite{fl}. In particular, it is canonically isomorphic to $\mathrm{H}(\T^2;\Z)$.
 \medskip

\noindent{\bf Grading:}  Each chain complex $(\CF(\psi_1; c), \partial_J)$ above has a relative $\Z$-grading and the boundary operator decreases degrees by one. For the case $c=0$, the grading 
can be set by using the usual Conley-Zehnder index  of contractible fixed points (which is well-defined since $c_1(\T^2) = 0$.) In particular,   if $\psi_t$ is a Hamiltonian isotopy we have  
\begin{equation}
\label{hfham}
\HF_*(\psi_1; 0) = \mathrm{H}_{*+1}(\T^2;\Z).
\end{equation}
For a general class $c \in  \mathrm{H}_1(\T^2;\Z)$ the (relative) grading of $(\CF_*(\psi_1; c), \partial_J)$ is again determined by the Conley-Zehnder index and the overall shift can be fixed by choosing a homotopy class of symplectic trivializations of $z^*(T\T^2)$ where $z \colon S^1 \to \T^2$ is a smooth representative of $c$.

\medskip

\noindent{\bf Extension to all smooth isotopies:}
The property of invariance under Hamiltonian isotopy allows one to also define the Floer homology for any  smooth symplectic isotopy $\widetilde{\psi}_t$ of $(\T^2, \Omega).$ 
One simply  sets 
\begin{equation*}
\label{ }
\HF_*(\widetilde{\psi}_1) = \HF_*(\widetilde{\psi}_1 \circ \phi) 
\end{equation*}
where $\phi$ is a Hamiltonian diffeomorphism for which the fixed points of  $\widetilde{\psi}_1 \circ \phi$ are nondegenerate. For example, if $\widetilde{\psi}_t = id$ for all $t \in [0,1]$ we can perturb by the Hamiltonian flow of a $C^2$-small Morse function to
obtain
\begin{equation}
\label{hfidentity}
\HF(id) = \HF(id;0) = \mathrm{H}(\T^2;\Z).
\end{equation}

\medskip

\noindent{\bf Dichotomy:} Finally we recall the following well known alternative for the Floer homology whose proof we include for the sake of completeness.

\begin{prop}\label{hom}
Either $\psi_1$ is a Hamiltonian diffeomorphism in which case $\HF(\psi_1) = \HF(\psi_1; 0)$ and
 $\HF_*(\psi_1; 0)= \mathrm{H}_{*+1}(\T^2;\Z)$, or 
$\HF(\psi_1)$ is trivial.
\end{prop}

\begin{proof}
This can be derived easily using the Flux homomorphism. 
Let $\widetilde{\mathrm{Symp}}_0(\T^2, \Omega)$ denote the universal cover of $\mathrm{Symp}_0(\T^2, \Omega)$, the identity component of the 
space of symplectic diffeomorphisms  of $(\T^2, \Omega)$. The points of $\widetilde{\mathrm{Symp}}_0(\T^2, \Omega)$ are of the form $[\psi_t]$ where $\psi_t$ is a symplectic isotopy 
starting at the identity and $[\psi_t]$ is the homotopy class of $\psi_t$ relative its endpoints. The {\it flux homomorphism} 
$$\Ff \colon \widetilde{\mathrm{Symp}}_0(\T^2, \Omega) \to   \mathrm{H}^1(\T^2; \R) $$ is then defined by 
$$\Ff([\psi_t]) = \int_0^1[\vartheta_t] \, dt , $$ 
where  $\vartheta_t  = -i_{X_t} \Omega$ and  $X_t$ is the vector field 
generating $\psi_t$. Besides the fact that is indeed  a homomorphism  (where the target  $\mathrm{H}^1(\T^2; \R)$ is identified with
$\mathrm{Hom}(\pi_1(\T^2), \R)$),  the other crucial property of $\Ff$ it that its kernel consists of the classes in $\widetilde{\mathrm{Symp}}_0(\T^2, \Omega)$ which can be represented
by Hamiltonian isotopies, \cite{ms}.

Without loss of generality we may assume that  $\Omega = d\theta_1 \wedge d\theta_2$ where $\theta_1, \theta_2 \in \R/ \Z$ are  global angular coordinates on $\T^2$. Let $a_1$ and $a_2$
be the standard generators of $\pi_1(\T^2)$ corresponding to these coordinates. The flux $\Ff([\psi_t])$ is then specified by the 
two numbers $A_1 = \Ff([\psi_t])(a_1)$ and $A_2 = \Ff([\psi_t])(a_2)$.

Consider now the symplectic isotopy  
\begin{equation*}
\label{ }
\mathbf{S}_{t}(\theta_1, \theta_2) = \left(\theta_1 +tA_1, \theta_2   + t A_2 \right).
\end{equation*}
The flux of $[\mathbf{S}_t]$ is equal to that of $[\psi_{t}]$. Since $\Ff$ is a homomorphism, we have
$$\Ff([(\psi_{t})^{-1} \circ \mathbf{S}_{t} ]) = 0 \in \mathrm{H}^1(\T^2; \R)).$$
The characterization of the kernel of $\Ff$ then implies that 
$[(\psi_{t})^{-1} \circ \mathbf{S}_{t} ] = [\phi^t_{G}]$
for some Hamiltonian flow $\phi^t_{G}$. In particular, we have
\begin{equation*}
\label{simple}
 \mathbf{S}_{1}  = \psi_1 \circ \phi^1_{G}.
\end{equation*}
By the invariance of Floer homology under Hamiltonian isotopies, this yields
$$\HF(\psi_1) = \HF(\mathbf{S}_{1}).$$ The Floer homology of $\mathbf{S}_{1}$ is now easy to compute. If $A_1=A_2 =0 \mod 1$, then
$\mathbf{S}_{1}$ is the identity map, and by \eqref{hfham} and \eqref{hfidentity} we have 
\begin{equation*}
\label{ }
\HF(\mathbf{S}_{1}) = \HF(\mathbf{S}_{1}; 0),
\end{equation*}
and
\begin{equation*}
\label{ }
 \HF_*(\mathbf{S}_{1}; 0)= \mathrm{H}_{*+1}(\T^2;\Z).
\end{equation*}
Otherwise, $\mathbf{S}_{1}$ has no fixed points and hence $\HF(\mathbf{S}_{1})$ is trivial. The result follows.
\end{proof}

\section{Proof of Theorem \ref{main}}

Let $\phi$ be a  Hamiltonian diffeomorphism of $S^2$ with finitely many periodic points.  It suffices to prove Theorem \ref{main}
for any iteration of $\phi$. Using the freedom to chose this iteration we may assume that the periodic points of $\phi$ are all fixed points (have period one).  The iteration formula \eqref{iter} implies that we may also assume that the mean index of any fixed point of $\phi$ is either irrational or is equal to zero modulo $4$.  In particular, we have
\begin{equation*}
\label{ }
\Per(\phi) = \Fix(\phi) = \{ p_1, \dots , p_l, r_1, \dots, r_m\} , 
\end{equation*}
where  for $j =1, \dots, l$ the mean indices  $\Delta(p_j)$ are irrational, and for $j =1, \dots, m$ we have $\Delta(r_j) =0 \mod 4$.

\subsection{Resonance and periodic points with irrational mean indices}

The starting point for the proof of Theorem \ref{main} is to show that the number of fixed points of $\phi$ with irrational mean indices is at least two. 
To prove this we require  the theory of resonance relations for Hamiltonian 
diffeomorphisms developed in \cite{gk}.  We now present a condensed version of these results, keeping only those features which are relevant to 
the task at hand.

As before, let $(M,\omega)$ be a closed symplectic manifold of dimension $2n$ with minimal Chern number $N$. 
Suppose also that $(M, \omega)$ is both weakly monotone (see, for example,
\cite{MS}) and rational.  A Hamiltonian diffeomorphism is said to be \emph{perfect} if it has finitely 
many contractible periodic
points all of which are fixed points. Let $\varphi$ be a perfect Hamiltonian diffeomorphism
of $(M, \omega)$ and let $\Delta_1,\ldots, \Delta_m$ be the collection of irrational mean indices of the contractible fixed
points of $\varphi$ (as described in Section \ref{indexmod} these are defined modulo $2N$).  A \emph{resonance relation}
for $\varphi$ is a vector $\va=(a_1,\ldots,a_m)\in\Z^m$ such that
$$
a_1\Delta_1+\ldots+a_m\Delta_m=0\mod 2N.
$$
The set of resonance relations of $\varphi$ forms a free abelian group $\CR=\CR(\varphi)\subset \Z^m$. 

\begin{thm}({\cite{gk}})
\label{thm:res}
Assume that $n+1\leq N<\infty$. 
\begin{itemize}

\item[(i)] Then $\CR\neq 0$, i.e., the irrational mean indices $\Delta_i$ satisfy at least
  one non-trivial resonance relation.

\item[(ii)] If there is only one resonance relation, i.e.,
  $\rk\CR=1$, then it has a generator of the form  $r \va=(ra_1,\ldots,ra_m)$, 
where $a_i\geq 0$ for all $i$,
$$
\sum a_i \leq \frac{N}{N-n},
$$
and $r$ is the smallest natural number such that the
mean index of each fixed point of $\phi^r$ is either irrational or is equal to zero modulo $2N$.\footnote{In \cite{gk}, this second statement is stated for the collection all nonzero mean indices, in which case one can take $r=1$. The formulation here is described in Remark 2.1 of \cite{gk}.} 
\end{itemize}

\end{thm}

Returning to the proof of Theorem \ref{main}, we derive the following consequence for our Hamiltonian diffeomorphism $\phi$.

\begin{cor}
\label{two}
At least two of the fixed points of $\phi$, say $P$ and $Q$, have irrational mean indices, i.e., are strongly nondegenerate elliptic fixed points. Moreover, if $P$ and $Q$ are the only fixed points of $\phi$, then their irrational mean indices 
 $\Delta(P)$ and $\Delta(Q)$ satisfy
 $$
 \Delta(P) + \Delta(Q) = 0 \mod 4.
 $$
\end{cor}

\begin{proof}
The sphere is weakly monotone and rational and its minimal Chern number is two ($N=2=n+1$). Hence, Theorem \ref{thm:res} applies and part (i) implies the first assertion of the corollary since there must be at least two fixed points of $\phi$ with irrational mean indices in order for a single non-trivial resonance relation to exist. 

Suppose that $\phi$ has exactly two fixed points, $P$ and $Q$. By part (i) both $P$ and $Q$ have irrational mean indices. These indices can not satisfy two independent resonance relations, otherwise they would be  the unique solutions (modulo $4$) of a linear system  with integer coefficients and hence would be rational. So, in this case,  $\rk\CR=1$ and the conclusion of part (ii) of Theorem \ref{thm:res} holds where the natural number $r$ can be taken to be $1$. This immediately implies the second assertion of Corollary \ref{two}.
\end{proof}

\subsection{An assumption and two paths to a contradiction}

By Corollary \ref{two} and our previous choices we now have
\begin{equation*}
\label{ }
\Per(\phi) = \Fix(\phi) = \{ P, Q, p_3, \dots , p_l, r_1, \dots, r_m\}.
\end{equation*}

\begin{lem}\label{deg}
If there is no fixed point of $\phi$ of type $r_j$ (with $\Delta(r_j) =0\mod 4$) then the points $P$ and $Q$ are the only fixed points of $\phi$.
\end{lem}

\begin{proof}
Arguing by contradiction assume that $\Fix(\phi) = \{ P, Q, p_3, \dots , p_l\}$ where $l>2$ and the $\Delta(p_j)$ are all irrational. The fixed points of $\phi$ are then all nondegenerate  and elliptic and so their  topological indices are even. By the Lefschetz fixed point theorem we would then have the Euler characteristic of $S^2$ equal to $l >2$.
\end{proof}

By Corollary  \ref{two} and Lemma \ref{deg} we will be  done if we can show that it is impossible for  $\phi$ to have even one fixed point, $r_1$, with $\Delta(r_1) =0 \mod 4$.  {\bf Arguing by contradiction, we assume that such a point  exists}.
At this point the path to a contradiction splits into two; the first corresponding to the case when at least one of the $r_j$ is degenerate, and the second to the case when all the $r_j$ are nondegenerate.

\subsection{Path 1: one of the $r_j$ is degenerate}

Assume that $\phi$ as above has a fixed point, say  $R=r_1$, which is degenerate (and satisfies $\Delta(R) = 0 \mod 4$). 

\subsubsection{A useful generating Hamiltonian}\label{staticham} We now choose a generating Hamiltonian $H \colon \R/\Z \times S^2 \to \R$ for $\phi$
such that  $P$ and $R$ are both static fixed points of the flow of $H$, that is  $\phi=\phi^1_H$ and both $P$ and $R$ are fixed points of $\phi^t_H$ for all $t \in \R$.  
We begin with any Hamiltonian $G$ generating $\phi$.  Let $u_R \colon \mathbb{D}^2 \to S^2$  be a smooth spanning disc for $\phi^t_G(R)$. As described in Section 9 of  \cite{SZ}  (see also Section 5.1 of \cite{Gi:conley}), one can use this disc to construct a  contractible loop of Hamiltonian diffeomorphisms,  $\gamma_1^t$, such that $\gamma_1^t \circ \phi^t_G(R)= R$ for all $t \in \R$, and $\gamma^t_1$  is supported in an arbitrarily small neighborhood of the image of $u_R$ (which might be all of $S^2$).  The curve  $\gamma_1^t \circ \phi^t_G(P)$ does not pass through $R$ and is contractible in its complement. Hence, we can choose  a spanning disc  $u_P$ for  $\gamma_1^t \circ \phi^t_G(P)$ whose image doesn't contain $R$. Using it, as above, we can then  
construct a contractible loop of Hamiltonian diffeomorphisms,  $\gamma_2^t$, which is trivial in a neighborhood of $R$, and satisfies $\gamma_2^t \circ\gamma_1^t \circ \phi^t_G(P)= P$ for all $t \in \R$. Let $H$ be the unique generating Hamiltonian of the Hamiltonian path $\gamma_2^t \circ\gamma_1^t \circ \phi^t_G$ such that $H(t,R) =0$ for all $t \in \R/\Z$.  By reparameterizing the path $\gamma_2^t \circ\gamma_1^t \circ \phi^t_G$ we may also assume that $H$ vanishes when $t$ is within some small fixed distance, say $0<\delta_H \ll 1$, of  $0\in \R/\Z$.

\subsubsection{A generic perturbation of $\phi^k$}

For a $k \in \N$, the Hamiltonian diffeomorphism $\phi^k$ is generated by the Hamiltonian 
\begin{equation*}
\label{ }
H_k(t,p)=kH(kt, p).
\end{equation*}
More precisely, we have $\phi^t_{H_k} = \phi^{kt}_H$ for all $t \in \R$. Note that $P$ and $R$ are still static fixed points of 
the flow of $H_k$ and, by \eqref{iter} we have 
\begin{equation}
\label{kbulletp}
\Delta(P; \phi^t_{H_k}, [\xi^k]) = k\Delta(P; \phi^t_{H}, [\xi]).
\end{equation}
and, similarly 
\begin{equation}
\label{kbulletr}
\Delta(R; \phi^t_{H_k}, [\xi^k])  = k\Delta(R; \phi^t_{H}, [\xi])
\end{equation}
for any choice of the class $[\xi]$. 

\begin{lem}\label{k-pert}
For each $k \in \N$ there is a neighborhood $U_k$ of $R$ and a Hamiltonian flow $\widetilde{\phi}_{k,t}$ which
 is arbitrarily $C^{\infty}$-close to $\phi^t_{H_k}$, is equal to $\phi^t_{H_k}$ outside of $U_k$, and whose fixed point set has the form
 \begin{equation*}
\label{ }
\Fix(\widetilde{\phi}_{k,1}) = \{ P, Q, p_3 \dots , p_l, R,R_1, \dots, R_d , r_2,  \dots, r_m\},
\end{equation*}
where
\begin{itemize}
  \item[(i)]  $R$ is a fixed point of $\widetilde{\phi}_{k,t}$ for all $t$, an elliptic fixed point of $\widetilde{\phi}_{k,1}$, and 
\begin{equation*}
\label{ }
\Delta(R; \widetilde{\phi}_{k,t}, [\xi^k]) = k\Delta(R; \phi^t_{H}, [\xi]) + \lambda/\pi
\end{equation*}
where $[\xi]$ is any class of symplectic trivializations, and $\lambda>0$ is arbitrarily close to $0$. 
  \item[(ii)] the $R_j$ are all contained in $U_k$. They are nondegenerate and $\Delta(R_j; \widetilde{\phi}_{k,t}, [\xi^k])$ is arbitrarily close to  $k\Delta(R; \phi^t_{H}, [\xi])$ for $j =1, \dots, d$ and any choice of $[\xi]$.
  \item[(iii)] none of the $\widetilde{\phi}_{k,t}$ trajectories of the remaining fixed points of $\widetilde{\phi}_{k,1}$ enter $U_k$. 
\end{itemize}

\end{lem}

\begin{proof}
Besides some simple manipulations we will require only the following generic transversality result for Hamiltonian diffeomorphisms:  Let $\phi^1_F$ be a Hamiltonian diffeomorphism 
of a symplectic manifold $(M, \omega)$ and $U$ an open subset of $M$ whose boundary is smooth and contains 
no fixed point of $\phi^1_F$. Then there is a Hamiltonian $\widetilde{F}$  arbitrarily $C^{\infty}$-close to $F$ which equals $F$
in the complement of $S^1 \times U$  and whose  fixed points in $U$ are all nondegenerate. 

Now, choose a Darboux ball $U_k$ around $R$ such that none of the $1$-periodic trajectories of the Hamiltonian flow of $H_k$, other than that through $R$, enter $U_k$. Let $(x,y)$ be the Darboux coordinates in $U_k$.
By our choice of $H$, it follows from the definition of $H_k$ that it vanishes for $t \in [1-\delta_H/k, 1]$.
Let $G$ be a small function supported in $U_k$ which equals $\frac{\lambda_0}{2\pi}(x^2 +y^2)$ near $R$ for a $\lambda_0$ which is a(n) (arbitrarily) small positive number. Let $\kappa\colon \R \to[0,1]$ be a smooth bump function 
such that $\kappa(t)=1$ for $t \in [1-3\delta_H/4k, 1- \delta_H/4k]$ and $\kappa$ vanishes outside  $(1-\delta_H/k, 1)$. Viewing $\kappa$ as a $1$-periodic function, we set $G'(t,p)= \kappa(t)G(p)$ and let $\phi_t'$ be the Hamiltonian flow of $G' +H_k$. Clearly $R$ is still a static fixed point of $\phi_t'$ and since the flows of $G'$ and $H_k$ are supported in disjoint time domains we have 
\begin{equation*}
\label{ }
\Delta(R; \phi^t_{G'+H_k}, [\xi^k]) = \Delta(R; \phi^t_{H_k}, [\xi^k]) + \Delta(R; \phi^t_{G'}, [\xi^k]) = k\Delta(R; \phi^t_{H}, [\xi]) + \lambda/\pi
\end{equation*}
where $\lambda= \lambda_0 \int_0^1 \kappa(t) \, dt$. This settles the  assertion (i) of the lemma.

In appropriate coordinates, the linearization of $\phi'_1$ at $R$ is rotation by $\lambda$ radians. Hence there are no fixed points of $\phi'_1$ in some Darboux ball $V$ around $R$ in $U_k$. Using the fact above, we can then perturb $G'+H_k$ in $S^1 \times (U_k \smallsetminus V)$ to obtain a Hamiltonian  $\widetilde{F}$  whose fixed points $R_1, \dots, R_d$ in $U_k \smallsetminus V$ are all nondegenerate. Setting $\widetilde{\phi}_{k,t} = \phi^t_{\widetilde{F}}$ we are done. In particular,  the continuity property of the mean index described in Section \ref{continuity} implies that each $\Delta(R_j; \widetilde{\phi}_{k,t}, [\xi^k])$ is arbitrarily close to $\Delta(R; \phi^t_{H_k}, [\xi^k])$ and hence $k\Delta(R; \phi^t_{H}, [\xi])$.  Thus condition (ii) is satisfied. Our choice of $U_k$ ensures that condition (iii) is also satisfied.ÊÊÊ
\end{proof}

\subsubsection{Completing the restriction of $\widetilde{\phi}_{k,1}$ to $S^2 \smallsetminus \{P,R\}$}\label{complete}
The symplectic manifold $(S^2 \smallsetminus \{P,R\}, \omega)$ is symplectomorphic to the open cylinder $(-1,1) \times \R/2\pi \Z$
equipped with the symplectic form $dz \wedge d\theta$. We now show that the map 
$\widetilde{\phi}_{k,1}$ can be completed to an area preserving diffeomorphism $\overline{\phi}_{k}$ of the closed cylinder $[-1,1] \times \R/2\pi \Z$, where $\overline{\phi}_{k}$ acts on the boundary circles, $\Gamma_P = \{1\} \times  \R/2\pi \Z$ and  $\Gamma_R = \{-1\} \times  \R/2\pi \Z$, as the rotation by  $\pi \Delta(P; \widetilde{\phi}_{k,t}, [\xi])$ and $\lambda$, respectively, for any choice of the class $[\xi]$. 
 
In general, if $X$ is an elliptic fixed point of a symplectic diffeomorphism $\psi_1$ of $S^2$ which is isotopic to the identity, then
the eigenvalues of $D(\psi_1)_{{X}}$ are $e^{\pm i \pi \Delta(X, \psi_t, [\xi])}$, which are independent of the  choice of the class $[\xi]$.
Hence, the eigenvalues of $D(\phi_{k,1})_R$ are $e^{\pm i \pi \lambda}$. If $B(\epsilon)$ is the open ball in $\R^2$ of radius $\epsilon>0$  centered at the origin, then for sufficiently small 
$\epsilon$ there is a symplectic embedding $A_R \colon B(\epsilon) \to (S^2, \omega)$ such that $A_R(0) =R$ and 
$$
D( {A_R}^{-1} \circ \widetilde{\phi}_{k,1} \circ A_{R})_0 =
\begin{pmatrix}
  \cos \lambda    & -\sin \lambda   \\ 
  \sin \lambda   &  \cos \lambda
\end{pmatrix}.
$$
Now consider the map $(r, \theta) \mapsto (\rho= r^2/2, \theta)$ which  takes $(B(\epsilon) \smallsetminus 0 , r dr \wedge d \theta)$ to $\left( (0, \epsilon^2) \times \R / 2 \pi \Z , d \rho \wedge d\theta \right)$. It follows from the linearization above, that in  $(\rho, \theta)$ coordinates the map
${A_R}^{-1} \circ \widetilde{\phi}_{k,1} \circ A_{R}$ extends to the boundary circle $\{0\} \times \R / 2 \pi \Z$ as the map
$$
(\rho, \theta) \mapsto (\rho, \theta+ \lambda).
$$
Applying the same procedure near $P$ we get the desired map $\overline{\phi}_{k}$.

Since $R$ and $P$ are static fixed points of the flow $\widetilde{\phi}_{k,t}$, we can complete each of the maps $\widetilde{\phi}_{k,t}$ to an area preserving diffeomorphism  $\overline{\phi}_{k,t}$ of the same closed cylider. 
(The restriction of $\overline{\phi}_{k,t}$ to the boundary circle $\Gamma_P$  ($\Gamma_R$) will only be a rotation 
when $P$ ($R$) is an elliptic fixed point of $\widetilde{\phi}_{k,t}$, but this is inconsequential since the boundary circles are always invariant.) In this way we obtain a smooth isotopy of area preserving diffeomorphisms  $\overline{\phi}_{k,t}$ starting at the identity and ending at $\overline{\phi}_{k}$.

\subsubsection{Transfer of dynamics to the torus}\label{transfer}

As in Arnold's famous argument from Appendix 9 of \cite{ar} in support of  his conjectured lower bound for the number of fixed points of Hamiltonian diffeomorphisms, we now extend the map $\overline{\phi}_{k}$ to the torus  formed by gluing two copies of the domain cylinder  $[-1,1] \times \R/2\pi \Z$ along their common boundaries.
In fact, as in \cite{ar},  we first insert two narrow {\it connecting cylinders} along the boundary circles to obtain a symplectic torus $(\T^2, \Omega)$ of total symplectic area  $(8 + 4\tau) \pi$ where each connecting cylinder is 
symplectomorphic to $[0,\tau] \times \R/ 2\pi \Z$. This allows us to extend the  map $\overline{\phi}_{k}$  to an area preserving map  $\psi_k$ of  $(\T^2, \Omega)$ which agrees with $\overline{\phi}_{k}$ on the two large cylinders, and is defined on the connecting cylinders so that  the overall map is smooth. Since $\overline{\phi}_{k}$ has no fixed points on the boundary of its domain, we may also assume (again as in \cite{ar}) that  no new fixed points are introduced in the connecting cylinders. Hence,
$\Fix(\psi_k)$ consists of  two copies of $\Fix(\overline{\phi}_{k})$, which we denote by
 \begin{equation*}
\label{ }
\Fix(\overline{\phi}^{\pm}_{k})=  \{ Q^{\pm}, p^{\pm}_3 \dots , p^{\pm}_l,R^{\pm}_1, \dots, R^{\pm}_d , r^{\pm}_2,  \dots, r^{\pm}_m\} .
\end{equation*}

\subsubsection{The contradiction at the end of Path 1.}

The isotopy $\overline{\phi}_{k,t}$ induces a smooth isotopy $\psi_{k,t}$ from the identity to $\psi_k$.  Hence, the Floer homology of $\psi_k$ is well defined. Now, there are two fixed points of $\psi_k$, $Q^{\pm}$, corresponding to the fixed point $Q$ of $\phi$. As described below, the following result concerning the role of $Q^+$ in $\HF(\psi_k)$ contradicts Proposition \ref{hom}.

\begin{prop}\label{not}
If  $k$ is sufficiently large  then $Q^+$ represents a nontrivial class in $\HF(\psi_k)$,  and if $Q^+$ is contractible then  the degree of the class $[Q^+]$ is  greater than one
in absolute value.
\end{prop}

\begin{proof}

Fix a class $[{\xi}]$ of symplectic trivializations of $TS^2$ along $\phi^t_{H}(Q)$. The resulting class $[{\xi}^k]$ of symplectic trivializations of $TS^2$ along $\phi^t_{H_k}(Q)$ then determines an equivalence class $[\xi^k_{+}]$ of symplectic trivializations of $T\T^2$ along $\psi_{k,t}(Q^+)$.
Let $X^{\pm}$ be any fixed point of $\psi_k$ in the same homotopy class as $Q^+$, where $X$ is the corresponding fixed point of $\widetilde{\phi}_{k,1}$. Since  the Floer boundary operator decreases degrees by one, to prove the first assertion of Proposition \ref{not} it suffices  to show that for $k$  large enough we have either 
\begin{equation*}
\label{ }
\Delta(X^{\pm}; \psi_{k,t}, [\xi^k_+]) = \Delta(Q^+; \psi_{k,t}, [\xi^k_+])
\end{equation*}
or 
\begin{equation*}
\label{ }
|\Delta(X^{\pm}; \psi_{k,t}, [\xi^k_+]) - \Delta(Q^+; \psi_{k,t}, [\xi^k_+])|>3.
\end{equation*}
More precisely, by Lemma \ref{determine} the first equality implies $$\mu(X^{\pm}; \psi_{k,t}, [\xi^k_+]) = \mu(Q^+; \psi_{k,t}, [\xi^k_+]).$$ Whereas, if the alternative inequality holds and $\widetilde{X}^{\pm}$ is a nondegenerate fixed point obtained from $X$ by perturbation, then it follows from the continuity property of the mean index and \eqref{closer} that
\begin{equation*}
\label{ }
|\mu(\widetilde{X}^{\pm}; \psi_{k,t}, [\xi^k_+]) - \mu(Q^+; \psi_{k,t}, [\xi^k_+])|>1.
\end{equation*}
Since there are only finitely many points $X^{\pm}$, for any $k$,  it will follow that for large enough $k$ the point $Q^+$ is in the kernel of the corresponding Floer boundary map and is not in its image thus establishing the first assertion of Proposition \ref{not}.

\noindent{\it Case 1: $X$ corresponds to a fixed point of $\phi^k$, i.e., $X \neq R^{\pm}_j$.}
By construction, we have 
$$
\Delta(Q^+; \psi_{k,t}, [\xi^k_+])=\Delta(Q; \phi^t_{H_k}, [\xi^k])
$$
and the iteration formula \eqref{iter} then implies that 
\begin{equation}
\label{fun}
\Delta(Q^+; \psi_{k,t}, [\xi^k_+]) =k\Delta(Q; \phi^t_H, [\xi]).
\end{equation}
Similarly, 
\begin{equation}
\label{hog}
\Delta(X^{\pm}; \psi_{k,t}, [\xi^k_+]) = k\Delta(X; \phi^t_H, [\xi]).
\end{equation}
If $\Delta(X; \phi^t_H, [\xi])=\Delta(Q; \phi^t_H, [\xi])$ then we are done as this would imply 
\begin{equation*}
\label{ }
\Delta(X^{\pm}; \psi_{k,t}, [\xi^k_+]) = \Delta(Q^+; \psi_{k,t}, [\xi^k_+]).
\end{equation*}
If  $\Delta(X; \phi^t_H, [\xi]) \neq \Delta(Q; \phi^t_H, [\xi])$, then equations \eqref{fun} and \eqref{hog} yield
\begin{equation*}
\label{ }
|\Delta(X^{\pm}; \psi_{k,t}, [\xi^k_+]) - \Delta(Q^+; \psi_{k,t}, [\xi^k_+])|  =  k|\Delta(X; \phi^t_H, [\xi]) - \Delta(Q; \phi^t_H, [\xi])|.
\end{equation*} 
and for sufficiently large $k$ we have 
\begin{equation*}
\label{ }
|\Delta(X^{\pm}; \psi_{k,t}, [\xi^k_+]) - \Delta(Q^+; \psi_{k,t}, [\xi^k_+])|>3,
\end{equation*}
as desired.

\noindent{\it Case 2: $X=R^{\pm}_j$.} 
Since $\Delta(R^{\pm}_j; \psi_{k,t}, [\xi^k_+]) = \Delta(R_j; \widetilde{\phi}_{k,t}, [{\xi}])$, it follows from
Lemma \ref{k-pert}  that $\Delta(R^{\pm}_j; \psi_{k,t}, [\xi^k_+])$ is arbitrarily close to
$ k \Delta(R; \phi^t_H, [{\xi}]).$
By assumption, $\Delta(R; \phi^t_H, [{\xi}])$ is an integer (multiple of four) and hence not equal to 
the irrational number $\Delta(Q; \phi^t_H, [\xi_{Q}])$. Arguing as in Case 1, we see that 
for sufficiently large $k$ we have
\begin{equation*}
\label{ }
|\Delta(R^{\pm}_j; \psi_{k,t}, [\xi^k_+]) -\Delta(Q^+; \psi_{k,t}, [\xi^k_+])| >3.
\end{equation*}

Finally, we settle  the second assertion of Proposition \ref{not}. If $Q^+$ is a contractible fixed point  of $\psi_k$ than 
$\phi^t_H(Q)$ is contractible in $S^2 \smallsetminus \{P,R\}$. We choose $[\xi]$ in this case so that it is determined by a spanning disc for 
$\phi^t_H(Q)$. Then the induced class $[\xi^k_+]$ determines the canonical grading of $\HF_*(\psi_k)$. As established above, we have
$$
\Delta(Q^+; \psi_{k,t}, [\xi^k_+]) = k\Delta(Q; \phi^t_H, [\xi]).
$$
Since $\Delta(Q; \phi^t_H, [\xi])$ is irrational and hence nonzero,  we therefore have 
$$ |\Delta(Q^+; \psi_{k,t}, [\xi^k_+])| >2$$  for large enough $k \in \N$.
For such $k$, inequality \eqref{closer} then yields $$|\mu(Q^+; \psi_{k,t}, [\xi^k_+])| >1.$$

\end{proof}

Now, Propositions \ref{hom} and \ref{not} can not both be true. The first assertion of Proposition \ref{not} together with 
Proposition \ref{hom} implies that $\psi_k$ must be a Hamiltonian diffeomorphism, in which case $\HF_d(\psi_k;0)$ 
must be  trivial when $|d|>1$. This contradicts the second assertion of Proposition \ref{not}.
Thus, $\phi$ can not have a degenerate fixed point.

\subsection{Path 2: all the $r_j$ are nondegenerate}

To begin we choose, as in Section \ref{staticham}, a generating Hamiltonian $H$ for $\phi$ such that, this time,  
$P$ and $Q$ are static fixed points of $\phi^t_H$. Fix a(n) (arbitrarily large) $k \in \N$ such that the fixed points of $\phi^k$ are 
also all nondegenerate. Following Section \ref{complete}, we can then 
complete the restriction of $\phi^k$ to $S^2 \smallsetminus \{P,Q\}$ to obtain a smooth area preserving map 
$\overline{\phi^k}$ of the closed cylinder $[-1,1] \times \R/2\pi \Z$ which acts on the boundary circles, $\Gamma_P = \{1\} \times  \R/2\pi \Z$ and  $\Gamma_Q = \{-1\} \times  \R/2\pi \Z$, as the (irrational) rotations by  $\pi \Delta(P; \phi^t_{H_k}, [\xi])$ and $\pi \Delta(Q; \phi^t_{H_k}, [\xi])$, respectively, for any choice of the classes $[\xi]$. Moreover, the flow $\phi^t_{H_k}$ again induces an isotopy $\overline{\phi_t^k}$ from the identity to $\overline{\phi^k}$. 

Proceeding as in Section \ref{transfer} we  extend the map $\overline{\phi^k}$ to the torus  formed by gluing together two copies of the domain cylinder  $[-1,1] \times \R/2\pi \Z$  to one another with two narrow  connecting cylinders in between. 
In this way we obtain an area preserving map  $\Psi_k$ of the symplectic torus  $(\T^2, \Omega)$ which agrees with $\overline{\phi^k}$ on the two large cylinders, and is defined on the connecting cylinders so that  the overall map is smooth and has no new fixed points. In particular, $\Fix(\Psi_k)$ consists of  two copies of $\Fix(\overline{\phi^k})$, which we denote by
 \begin{equation*}
\label{ }
\Fix^{\pm}(\overline{\phi^k})=  \{  p^{\pm}_{3} \dots , p^{\pm}_{l} , r^{\pm}_{1},  \dots, r^{\pm}_{m}\} .
\end{equation*}

The isotopy $\overline{\phi^k_t}$ induces a smooth isotopy $\Psi_{k,t}$ from the identity to $\Psi_k$ and so we can 
again consider the Floer homology $\HF(\Psi_k)$. The following result again contradicts Proposition \ref{hom}.

\begin{prop}\label{endpath2}
If $k$ is sufficiently large then no contractible fixed point of $\Psi_k$ has Conley-Zehnder index equal to one, and $r^+_1$ represents a nontrival class in $\HF(\Psi_k)$.
\end{prop}

\begin{proof}

Let $X^k$ be a contractible fixed point of $\Psi_k$ where $X$ denotes the corresponding fixed point of $\phi$. 
Since $\pi_2(\T^2)$ is trivial, all classes of symplectic trivializations determined
by spanning discs for $\Psi_{k,t}(X^k)$  yield the same values of the mean index and Conley-Zehnder index of $X^k$.
So, in what follows we denote these simply as $\Delta(X^k; \Psi_{k,t})$ and $\mu(X^k; \Psi_{k,t})$. Since $X^k$ is contractible, $X$ must admit a spanning disc with image in $S^2 \smallsetminus \{P,Q\}$. Let $\Delta(X; \phi^t_H)$ denote the mean index computed with respect to the corresponding class of trivializations along $\phi^t_H(X)$. By \eqref{iter} we have 
\begin{equation}
\label{ktimes}
\Delta(X^k; \Psi_{k,t}) = k \Delta(X; \phi^t_H).
\end{equation}

\noindent{\it Case 1: $X^k= p^{\pm}_{j}$.} Since $\Delta(p_j; \phi^t_H)$ is irrational, it follows from \eqref{ktimes} that  for large enough $k$ we have 
\begin{equation*}
\label{ }
|\Delta(p^{\pm}_{j}; \Psi_{k,t})| = k |\Delta(p_j;\phi^t_H)| > 2.
\end{equation*} 
By \eqref{closer} it then follows that for sufficiently large $k$ we have 
\begin{equation*}
\label{up}
|\mu(p^{\pm}_{j}; \Psi_{k,t})| >1.
\end{equation*}

\noindent{\it Case 2: $X^k= r^{\pm}_{j}$.} In this case $\Delta(r_j; \phi^t_H) = 0 \mod 4$. If $\Delta(r_j; \phi^t_H) \neq 0$
then we can argue as in the previous case to show that for sufficiently large $k$ we have 
\begin{equation*}
\label{up}
|\mu(r^{\pm}_{j}); \Psi_{k,t}| >1.
\end{equation*}
Otherwise, it follows from \eqref{ktimes} that $$\Delta(r^{\pm}_{j}; \Psi_{k,t}) = 0.$$ 
Since $r^{\pm}_{j}$ is nondegenerate, the strong form of \eqref{closer} applies and implies that 
\begin{equation*}
\label{down}
\mu(r^{\pm}_{j}; \Psi_{k,t}) = 0.
\end{equation*}
This settles the  first assertion of Proposition \ref{endpath2}. 

To approach the second, we first fix a class $[{\xi}]$ of symplectic trivializations of $TS^2$ along $\phi^t_{H}(r_1)$. This determines an equivalence class $[\xi^k_{+}]$ of symplectic trivializations of $T\T^2$ along $\Psi_{k,t}(r^+_1)$. Let $X^k$ be any fixed point of $\Psi_k$ in the same homotopy class as
$r_1^+$. Arguing as in Proposition \ref{not}, to prove the second assertion of Proposition \ref{endpath2} it suffices to show that for $k$ sufficiently large 
we have either 
\begin{equation*}
\label{ }
\Delta(X^k; \Psi_{k,t}, [\xi^k_+]) = \Delta(r_1^+; \Psi_{k,t}, [\xi^k_+])
\end{equation*}
or 
\begin{equation*}
\label{ }
|\Delta(X^k; \Psi_{k,t}, [\xi^k_+]) - \Delta(r_1^+; \Psi_{k,t}, [\xi^k_+])|>3.\footnote{A very simillar argument to the one which follows appears in \cite{gg3}.}
\end{equation*}

\noindent{\it Case 1: $X^k=p^{\pm}_j$.} By our construction of $\Psi_k$ and  \eqref{iter} we have 
$$
\Delta(r^+_1; \Psi_{k,t}, [\xi^k_+])  =k\Delta(r_1; \phi^t_H, [\xi])
$$
and 
$$
\Delta(p^{\pm}_j; \Psi_{k,t}, [\xi^k_+])  =k\Delta(p_j; \phi^t_H, [\xi]).
$$
Now $\Delta(r_1; \phi^t_H, [\xi]) = 0 \mod 4$ and  $\Delta(p_j; \phi^t_H, [\xi])$ is irrational, so for $k$ sufficiently large we have 
$$
| \Delta(r^+_1; \Psi_{k,t}, [\xi^k_+]) - \Delta(p^{\pm}_j; \Psi_{k,t}, [\xi^k_+])| > 3.
$$

\noindent{\it Case 2:  $X^k=r^{\pm}_j$.} In this case, 
\begin{equation*}
\label{ }
\Delta(r^+_1; \Psi_{k,t}, [\xi^k_+]) - \Delta(r^{\pm}_j; \Psi_{k,t}, [\xi^k_+]) =0 \mod 4.
\end{equation*}
If the mean indices are equal we are done. If not we can argue as in the previous case to show that for large $k$ they differ by more than $3$.

\end{proof}

This leads to the desired contradiction at the end of Path 2 as Proposition \ref{endpath2} contradicts Proposition \ref{hom}. In particular,  the first assertion of Proposition \ref{endpath2}  implies that $\Psi_k$ can not be Hamiltonian and the second assertion of Proposition \ref{endpath2} implies that the Floer homology $\HF(\Psi_k)$ is nontrivial. With this, the proof of Theorem \ref{main} is complete.

\end{document}